\documentclass[12pt]{article}
\usepackage[centertags]{amsmath}
\usepackage{amsfonts}
\usepackage{amssymb}
\usepackage{amsthm}
\usepackage{authblk}
\input{amssym.def}

\numberwithin{equation}{section}

\newtheorem{Definition}{Definition}[section]
\newtheorem{theorem}[Definition]{Theorem}
\newtheorem{lemma}[Definition]{Lemma}
\newtheorem{proposition}[Definition]{Proposition}

\title{\textbf{Generalized Larcombe-Fenessey invariants of matrix powers}}
\author[1]{Sajal Kumar Mukherjee}
\author[2]{Sanjay Mukherjee}
\affil[1]{Institute for Advancing Intelligence, TCG Crest, Kolkata 700091, India. shyamal.sajalmukherjee@gmail.com}
\affil[2]{School of Mathematical Sciences, National Institute of Science Education and Research, Bhubaneswar 752050, India\\
Homi Bhabha National Institute (HBNI), Training School Complex, Anushakti Nagar, Mumbai 400094, India.
sanjay.mukherjee@niser.ac.in}
\date{September 2021}

\begin{document}
\maketitle
\begin{abstract}
In this article, we have found significant generalization of the invariance properties of powers of matrices discovered by Larcombe, Fenessey and further explored by Zeilberger. Moreover, we found interesting new results exhibiting similar phenomena in a more general setup.      
\end{abstract}


\section{Introduction}
Study of identities involving matrices is one of the most prevalent cultures in mathematics. From the well known \textit{Cayley-Hamilton Theorem} to \textit{Grassmann-Pl\"{u}cker Relations}, various type of matrix identities are ubiquitous in the literature of mathematics, especially in combinatorics, algebra, geometry and topology. Recently in \cite{Larcombe} Larcombe found a surprising relation in 2$\times$2 matrices. He proved that,
\begin{theorem}
Let $A$ be a $2\times 2$ matrix with entries from $\mathbb{R}$ and $m$ be any positive integer. Then,
$$A_{12}(A^m)_{21}=A_{21}(A^m)_{12}.$$ 
\end{theorem}
\noindent Where $B_{ij}$ denote the $(i,j)-$th entry of the matrix $B.$

In \cite{LarFen} Larcombe and Fenessey proved the above result for $n\times n$ tri-diagonal matrices and later Zeilberger\cite{Z} gave an elegant bijective argument to conclude the aforesaid phenomena. In this paper, we have significantly generalized this result to general class of $n\times n$ matrices and studied the cases where the aforementioned phenomenon will hold in a much stronger form. For instance, the following type of matrices are few of the many, satisfying the Larcombe-Fenessey invariance phenomenon in more generality, which can be derived easily from the results obtained in this paper.
\paragraph{Example:}
\[
  \begin{bmatrix}
    a_{11} & b & b & b & \cdots & b \\
    c & a_{22} & a_{23} & a_{24} & \cdots & a_{2n} \\
    c & a_{23} & a_{33} & a_{34} & \cdots & a_{3n} \\
    c & a_{24} & a_{34} & a_{44} & \cdots & a_{4n} \\
    \vdots & \vdots & \vdots & \vdots & \ddots & \vdots \\
    c & a_{2n} & a_{3n} & a_{4n} & \cdots & a_{nn} \\
  \end{bmatrix}
,
  \begin{bmatrix}
    a_{11} & a_{12} & a_{13} & a_{14} & \cdots & a_{n1} \\
    a_{21} & a_{22} & 0 & \cdots & 0 & a_{21} \\
    a_{31} & 0 & a_{33} & \ddots & \vdots & a_{31} \\
    a_{41} & \vdots & \ddots & a_{44} & 0 & a_{41} \\
    \vdots & 0 & \cdots & 0 & \ddots & \vdots \\
    a_{n1} & a_{12} & a_{13} & a_{14} & \cdots & a_{nn} \\
  \end{bmatrix}.
\]
Before stating the main results of this paper, we define some terminologies. Let $A$ be a matrix of order $n$. We associate with $A$ a weighted-digraph $D(A)$ with $n$ vertices $\{1,2, \ldots, n\}$ and there is a directed edge $(i,j)$ from vertex $i$ to vertex $j$ if and only if $A_{i,j} \neq 0$ and on each directed edge $(i,j)$ put a weight $w((i,j)) = A_{i,j}.$  The matrix $A$ is called \emph{acyclic} if $D(A)$ does not contain any directed cycle of length $\geq 3.$ 

Now let us state our main results.
\begin{theorem}
Let $A$ be an $n \times n$ acyclic matrix with entries belonging to an integral domain $R$. Then for any given $k$ positive integers $m_1, m_2, \cdots, m_{k}$ and for any sequence of positive integers $1 \leq i_1, i_2, \cdots, i_k \leq n $, we have,  
$$(A^{m_1})_{i_1i_2}(A^{m_2})_{i_2i_3}\cdots (A^{m_{k}})_{i_ki_1} = (A^{m_1})_{i_2i_1}(A^{m_2})_{i_3i_2}\cdots (A^{m_{k}})_{i_1i_k}.$$
\end{theorem}
\noindent One can see that any matrix of order $2$ and tri-diagonal matrices are acyclic and obtain the Theorem $1.1$ and the Larcombe-Fenessey theorem as corollaries of the above theorem.

The following theorem also exhibits similar matrix identity in the context, where ``cycles" are present.
\begin{theorem}
Let $A$ be an $n \times n$ matrix with non-zero entries belonging to an integral domain $R$. If for any $1 \leq i < j \leq (n-1)$,  $A_{i,j}A_{j,j+1}A_{j+1,i}= A_{j,i}A_{j+1,j}A_{i,j+1} $, then for any given $k$ positive integers $m_1, m_2, \cdots, m_{k}$ and for any sequence of positive integers $1 \leq i_1, i_2, \cdots, i_k \leq n $, we have,  
$$(A^{m_1})_{i_1i_2}(A^{m_2})_{i_2i_3}\cdots (A^{m_{k}})_{i_ki_1} = (A^{m_1})_{i_2i_1}(A^{m_2})_{i_3i_2}\cdots (A^{m_{k}})_{i_1i_k}.$$
\end{theorem}

\section{Definitions and preliminaries}
An \emph{abstract simplicial complex}(or simply a \emph{simplicial
complex})\cite{Deo} $K$ is defined to be an ordered pair $(V, \mathcal{F})$, where $V$ is any finite set (we only
consider finite case here and from now on we will take the set V to be a finite subset
of positive integers) and $\mathcal{F} \subseteq P(V )$, the power set of $V$ satisfying the condition: if
$\sigma \in \mathcal{F}$ and $\tau \subseteq \sigma $, then $\tau \in \mathcal{F} $. The non-empty elements of $\mathcal{F}$ are called the \emph{faces} of the simplicial complex $K$. \emph{Dimension} of the face $\sigma$, denoted by $dim(\sigma)$ is defined to be $\vert \sigma \vert - 1$, where $\vert \sigma \vert $ is the
cardinality of $\sigma$. The zero-dimensional faces of $K$ are called \emph{vertices} of $K$. A $q-$dimensional face of $K$ is also called a $q - $simplex of $K$. The \emph{dimension} of a simplicial
complex is defined to be the maximum of the dimensions of its simplices. Note that,
graphs are one dimensional simplicial complexes.

Let $\sigma$ be a $q-$simplex of $K$. Two orderings of its vertex set are equivalent if they
differ by an even permutation. If $dim(\sigma) > 0$ then the orderings of the vertices of $\sigma$
fall into two equivalence classes. Each class is called an \emph{orientation} of $\sigma$. An \emph{oriented}
simplex is a simplex $\sigma$ together with an orientation of $\sigma$. Let $a_0, a_1, \cdots, a_q$ be an
ordering of the vertices of $\sigma$. Then we shall use the symbol  $[a_0, a_1, \cdots, a_q]$ to denote the
oriented simplex. Let $b_0, b_1, \cdots, b_q$ be another vertex-ordering, which differs
from the previous ordering by an odd permutation, then we write  $[a_0, a_1, \cdots, a_q] = - [b_0, b_1, \cdots, b_q]$.

Let $C_q(K)$ be the free
$\mathbb{Z}-$Module(equivalently the free abelian group) generated by all oriented $q-$simplices of $K$. Now for $q > 0$, we define a
homomorphism $\partial_q : C_q(K) \rightarrow C_{q-1}(K)$ (called $q-$\emph{th boundary operator}) as follows:
Let $\sigma = [v_0, v_1, \cdots, v_q]$ be an oriented $q-$simplex. Then,

$$\partial_q(\sigma):= \sum_{i=0}^q  (-1)^i [v_0, v_1, \cdots, \hat{v_i}, \cdots, v_q], $$ 
where $\hat{v_i}$ means that the vertex $v_i$ is omitted. An element of $ker(\partial_q)$ is called a \emph{q-cycle}. The \emph{q-th homology group} of $K$, denoted by $H_q(K)$ is defined to be the quotient group $\frac{ker(\partial_q)}{Im(\partial_{q+1})}$, which is known to be well defined.

Let us now define some digraph-theoretic notions\cite{Brualdi}, which will be needed subsequently. Let $G$ be a weighted-digraph (i.e. a digraph, where each directed edge $e$ is given a non-zero weight, denoted by $w(e)$). A \emph{walk} from vertex $u$ to vertex $v$ is a sequence $l$ of vertices $u=x_0, x_1, \ldots, x_m = v,$ such that $(x_i, x_{i+1})$ is a directed edge for each $i=0,1, \ldots, m-1.$ The \emph{reverse} of the walk $l$, denoted by $\Acute{l}$ is the sequence of vertices $v= x_m, x_{m-1}, \ldots, x_1=u.$ Note that the reverse of a walk need not always be a walk. Let $l_1$ be a walk from $u$ to $v$ defined as the sequence $u=p_0, p_1, p_2, \ldots, p_m =v $  and $l_2$ be a walk from $v$ to $w$ defined as the sequence $v= q_0, q_1, q_2, \ldots, q_n =w. $  The concatenation of $l_1$ and $l_2$, denoted by $l_1l_2$ is a walk from $u$ to $w$ defined as the sequence $u=p_0, p_1, p_2, \ldots, p_m = q_0, q_1, q_2, \ldots, q_n =w. $ In a weighted digraph, the \emph{weight of a walk l}, denoted by $w(l)$ is the product of the weights of all directed edges of the walk. For a digraph $G$, it's \emph{underlying undirected graph}, denoted by $\widehat{G}$, is defined to be the undirected simple graph $\widehat{G}$ with $V(\widehat{G})=V(G)$ and for any two different vertices $u\neq v$, $u$ is adjacent to $v$ in $V(\widehat{G})$ if and only if either $(u,v)$ or $(v,u)$ is a directed edge in $G.$

\section{Proofs of the theorems}

Before getting into the proofs of the theorems, we will develop a general framework and prove two technical results, which will yield the main theorems.
Let $K$ be a  simplicial complex. For some $c_q\in C_q(K)$ with $c_q=\sum_n a_n\sigma_n^q,$ $a_n\in \mathbb{Z}\setminus\{0\}$, we define $\bar{c_q}=\sum_n |a_n|\bar{\sigma_n^q}$, where $\bar{\sigma_n^q}=\sigma_n^q$ when $a_n>0$ and $\bar{\sigma_n^q}=- \sigma_n^q$ when $a_n<0$. Let $\bar{K}$ denote the set $\cup_{\sigma\in K}\{\sigma, - \sigma\}.$ 

Let $R$ be an integral domain and $f:\bar{K}\rightarrow R\setminus\{0\}$ be a function. For a $q$-chain $c_q\in C_q(K)$ with $c_q=\sum_n a_n\sigma_n^q,$ $a_n\in \mathbb{Z}\setminus\{0\}$ and $\sigma_i^q\neq \sigma_j^q$ whenever $i\neq j,$ the \emph{weight} of $c_q$, denoted by $W(c_q)$ is defined as  $W(c_q)(=W(\bar{c_q})):=\prod_nf(\bar{\sigma_n^q})^{|a_n|}$(Note that the condition $\sigma_i^q\neq \sigma_j^q$ whenever $i\neq j$ is necessary for the well-definedness of $W(c_q)$). Note that, the weight of the zero $q$-chain is $1$. It is easy to see that $W(c_q)=\prod_{n}W(\frac{a_n}{|a_n|}\sigma_n^q)^{|a_n|}$. Then we have the following.

\begin{proposition}
Let $S\in C_q(K)$ such that $S=\sum_is_ic^q_i$, where $c^q_1,c^q_2,
\ldots,c^q_m\in C_q(K)$, $s_i\in \mathbb{Z}\setminus\{0\}$ and $W(c_i^q)=W(-c_i^q)$ for all $i\in [m]$. Then $W(S)=W(-S).$     
\end{proposition}
\begin{proof}
Let us define $W_1:=\prod_{i}W(\frac{s_i}{|s_i|}c^i_q)^{|s_i|}$ and $W_2:=\prod_{i}W(-\frac{s_i}{|s_i|}c^i_q)^{|s_i|}.$ Clearly, $W_1=W_2$. We will show that, 
\begin{equation}
W(S)W_2=W(-S)W_1.\label{eq:3.2}    
\end{equation}

Let $\tau^q$ be a $q$-simplex appearing in at least one of the $c_i^q$'s. Let the coefficient of $\tau^q$ in $c_i^q$ be $t_i.$ Let the coefficient of $\tau^q$ in $S$ be $T_s$. Divide $[m]=A\cup B$ such that $s_it_i\geq 0$ if and only if $i\in A$ and $s_it_i<0$ if and only if $i\in B.$ Let $\mathcal{A}:=\sum_{i\in A}|s_it_i|$ and $\mathcal{B}:=\sum_{j\in B}|s_jt_j|.$ It is easy to see that $T_s=\sum_is_it_i=\mathcal{A}-\mathcal{B}.$ Hence the contribution of $\tau^q$ in $W_1$ is $f(\tau^q)^{\mathcal{A}}f(- \tau^q)^{\mathcal{B}}$ and the contribution of $\tau^q$ in $W_2$ is $f(\tau^q)^{\mathcal{B}}f(-\tau^q)^{\mathcal{A}}$. Without loss of generality, we assume that $T_s\geq 0.$ Then contribution of $\tau^q$ in $W(S)$ is $f(\tau_q)^{\mathcal{A}-\mathcal{B}}$ and the contribution of $\tau^q$ in $W(-S)$ is $f(- \tau^q)^{\mathcal{A}-\mathcal{B}}.$ 

So the contribution of $\tau^q$ to the L.H.S. of \eqref{eq:3.2} is $f(\tau^q)^{\mathcal{A}-\mathcal{B}}f(\tau^q)^{\mathcal{B}}f(-\tau_q)^{\mathcal{A}}=f(\tau^q)^{\mathcal{A}}f(-\tau^q)^{\mathcal{A}}$ and its contribution to the R.H.S. of \eqref{eq:3.2} is $f(-\tau^q)^{\mathcal{A}-\mathcal{B}}f(\tau^q)^{\mathcal{A}}f(-\tau^q)^{\mathcal{B}}=f(-\tau^q)^{\mathcal{A}}f(\tau^q)^{\mathcal{A}}.$ Hence the result follows.
\end{proof}

Now, suppose we have a finite sequence of non-zero $q$-chains of the form $\Gamma^q=\{b_1\sigma_1^q,b_2\sigma_2^q,\ldots,b_m\sigma_m^q\}$, where for each $i$, $\sigma_i^q$ is a $q$-simplex. \\ We define, $\mathcal{W}(\Gamma^q):=\prod_{i=1}^mW(\frac{b_i}{|b_i|}\sigma_i^q)^{|b_i|}$. \\ For the sequence $\Gamma^q=\{b_1\sigma_1^q,b_2\sigma_2^q,\ldots,b_m\sigma_m^q\}$, we define $-\Gamma_q :=\{-b_1\sigma_1^q,-b_2\sigma_2^q,\ldots,-b_m\sigma_m^q\}.$ \\ Hence $\mathcal{W}(-\Gamma^q)=\prod_{i=1}^mW(-\frac{b_i}{|b_i|}\sigma_i^q)^{|b_i|}.$ Then the following holds.
\begin{lemma}
Let $\Gamma^q=\{b_1\sigma_1^q,b_2\sigma_2^q,\ldots,b_m\sigma_m^q\}$ be a finite sequences of non-zero weighted simplices as above. Let $g^q=\sum_{i=1}^mb_i\sigma_i^q$ and $W(g^q)=W(-g^q).$ Then $\mathcal{W}(\Gamma^q)=\mathcal{W}(-\Gamma^q).$
\end{lemma}
\begin{proof}
We will show that, 
\begin{equation}
W(g^q)\mathcal{W}(-\Gamma^q)=W(-g^q)\mathcal{W}(\Gamma^q).\label{eq:3.3}  
\end{equation}

Let $\gamma^q$ be a $q$-simplex appearing at least once in $\Gamma^q$. Let the coefficient of $\gamma^q$ in $g^q$ be $u_g$. Without loss of generality, we assume $u_g\geq0.$ Let $b_{i_1},b_{i_2},\ldots,b_{i_l}$ be the positive coefficients of $\gamma^q$ in $\Gamma^q$ and $b_{j_1},b_{j_2},\ldots,b_{j_{\Tilde{l}}}$ be the negative coefficients of $\gamma^q$ in $\Gamma^q$. Let $L:=|b_{i_1}|+|b_{i_2}|+\cdots+|b_{i_l}|$ and $\Tilde{L}:=|b_{j_1}|+|b_{j_2}|+\cdots+|b_{j_{\Tilde{l}}}|.$ Then it trivially follows that $u_g=L-\Tilde{L}.$ So, the contribution of $\gamma^q$ in $W(g^q)$ and $\mathcal{W}(\Gamma_q)$ are  $f(\gamma^q)^{L-\Tilde{L}}$ and  $f(\gamma^q)^Lf(- \gamma^q)^{\Tilde{L}}$ respectively. Similarly, the contribution of $\gamma^q$ in $W(-g^q)$ and $\mathcal{W}(-\Gamma^q)$ are $f(-\gamma^q)^{L-\Tilde{L}}$ and $f(\gamma^q)^{\Tilde{L}}f(-\gamma^q)^L$ respectively.

So the contribution of $\gamma^q$ in L.H.S. of \eqref{eq:3.3} is $f(\gamma^q)^{L-\Tilde{L}}f(\gamma^q)^{\Tilde{L}}f(-\gamma^q)^L=f(\gamma^q)^Lf(-\gamma^q)^L$ and it's contribution to the R.H.S. of \eqref{eq:3.3} is $f(-\gamma^q)^{L-\Tilde{L}}f(\gamma^q)^Lf(-\gamma^q)^{\Tilde{L}}=f(-\gamma^q)^Lf(\gamma^q)^L.$ Hence, the result follows.
\end{proof}

With these two results in hand, let us get to the proof of the Theorems. We have found it convenient to prove Theorem $1.3$ before going to the proof of Theorem $1.2$, to avoid unnecessary repetition of some arguments. 
\begin{proof}
(Theorem $1.3$): In the digraph $D(A)$, let $L$ be the set of all closed walks of the form $L_1L_2 \ldots L_k$, where $L_1$ is a walk of length $m_1$ from $i_1$ to $i_2$, $L_2$ is a walk of length $m_2$ from $i_2$ to $i_3$ and so on and $L_k$ is a walk of length $m_k$ from $i_k$ to $i_1$. Let $\Acute{L} = \{\Acute{l} : l \in L\}$. Since the entries of $A$ are non-zero, each element of $\Acute{L}$ is a walk. Clearly the L.H.S is the sum of the weights of all walks in $L$, and the R.H.S is the sum of the weights of all walks in $\Acute{L}$. We claim that, $l \mapsto \Acute{l}$ is a weight preserving bijection between the L.H.S and R.H.S of Theorem $1.3$ for $l \in L.$ Suppose, $l \in L.$ Note that, if a loop $c$ occurs $m$ times in $l$, $c$ also occurs $m$ times in $\Acute{l}.$ So, $c$ contributes nothing to the quotient $\frac{w(l)}{w(\Acute{l})}.$  So, for the sake of simplicity, assume that $l$ does not contain any loops. Let $l$ be the sequence $i_1= x_0, x_1, \ldots, x_{r-1}, x_r=i_1.$ Now consider $K_n,$ the complete graph with $n$ vertices. The \emph{weight function} $f$ on the ordered $1$-simplices of $K_n$ is defined as $f([i,j]):= A_{ij}$. Now consider the finite sequence of ordered $1$-simplices in $K_n$ of the form $\Gamma^1 = \{[x_0,x_1], [x_1,x_2], \ldots, [x_{r-1}, x_r]\}.$ Clearly, $w(l)= \mathcal{W}(\Gamma^1)$ and $w(\Acute{l})= \mathcal{W}(-\Gamma^1).$ Let $g^1= [x_0,x_1]+[x_1,x_2]+ \cdots +[x_{r-1}, x_r].$ By the Lemma $3.2$, it is enough to show that $W(g^1)=W(-g^1).$ If $g^1=0,$ we are done. So, suppose that $g^1$ is a non-zero $1$-chain. Also note that $g^1$ is a $1$-cycle. Now it is an easy exercise in simplicial homology theory to show that the set $\mathcal{B}=\{[i,j]+[j,j+1]+[j+1,i], 1 \leq i < j \leq (n-1) \}$ is a basis of the homology group $H_1(K_n).$ So, $g^1$ is a $\mathbb{Z}$-linear combination of the elements of $\mathcal{B}.$ Now apply the Proposition $3.1$, taking $q=1$ and $S=g^1$ and noting that $W(b) = W(-b),$  $ \forall b \in \mathcal{B}$ to conclude the claim.
\end{proof}

Now we come to the proof of the Theorem $1.2$.
\begin{proof}
(Theorem $1.2$): In the digraph $D(A)$, let $L$ be the set of all closed walks of the form $L_1L_2 \ldots L_k$, where $L_1$ is a walk of length $m_1$ from $i_1$ to $i_2$, $L_2$ is a walk of length $m_2$ from $i_2$ to $i_3$ and so on and $L_k$ is a walk of length $m_k$ from $i_k$ to $i_1$. Let $\Acute{L} = \{\Acute{l} : l \in L\}$. It is easy to see that the reverse of each walk in $L$ is also walk. In fact, suppose that $l \in L.$ Since, $l$ is a closed walk and $D(A)$ does not contain any directed cycle of length $\geq 3,$ we conclude that whenever there is a directed edge $(i,j)$ in $l,$ there is also a directed edge $(j,i)$ in $l.$ This argument also shows that either both $L$ and $\Acute{L}$ are empty or both are non-empty. If both are empty, we are done, since both sides of the Theorem $1.2$ are $0$ (note that the L.H.S is the sum of the weights of all walks in $L$, and the R.H.S is the sum of the weights of all walks in $\Acute{L}$). So, suppose that both are non-empty. We claim that, $l \mapsto \Acute{l}$ is a weight preserving bijection between the L.H.S and R.H.S of Theorem $1.2$ for $l \in L.$ Suppose, $l \in L.$ As in the previous proof, we assume that $l$ does not contain any loop. Since $D(A)$ does not contain any directed cycle of length $\geq 3,$ the underlying undirected graph structure of any walk in $D(A)$ without any loop is a tree. In particular, the underlying undirected graph structure of $l$ is a tree, say $T$. Let $l$ be the sequence $i_1= x_0, x_1, \ldots, x_{r-1}, x_r=i_1.$ Now consider the finite sequence of ordered $1$-simplices in $T$ of the form $\Gamma^1 = \{[x_0,x_1], [x_1,x_2], \ldots, [x_{r-1}, x_r]\}.$ Clearly, $w(l)= \mathcal{W}(\Gamma^1)$ and $w(\Acute{l})= \mathcal{W}(-\Gamma^1).$ Let $g^1= [x_0,x_1]+[x_1,x_2]+ \cdots +[x_{r-1}, x_r].$ Since $l$ is a closed walk, $g^1$ is a $1$-cycle in $T.$ Since, $H_1(T)$ vanishes, both $g^1$ and $-g^1$ are zero. So, $W(g^1)=W(-g^1)=1$, and the claim follows by the Lemma $3.2.$
\end{proof}
\textbf{Remark:} The proof of Theorem 1.3 leads us to a more general result. Let $A$ be an $n\times n$ matrix with $A_{ij}=0$ if and only if $A_{ji}=0$ for $i\neq j$. Suppose $\mathcal{C}$ is a set of directed cycles(of length $\geq$ 3) in $D(A)$, forming a basis for $H_1(\widehat{D(A)})$ such that, $w(c)=w(\Acute{c})$ for all $c\in \mathcal{C}.$ Then, for any given $k$ positive integers $m_1, m_2, \cdots, m_{k}$ and for any sequence of positive integers $1 \leq i_1, i_2, \cdots, i_k \leq n $, we have,  
$$(A^{m_1})_{i_1i_2}(A^{m_2})_{i_2i_3}\cdots (A^{m_{k}})_{i_ki_1} = (A^{m_1})_{i_2i_1}(A^{m_2})_{i_3i_2}\cdots (A^{m_{k}})_{i_1i_k}.$$
\textbf{Acknowledgement:}

The first author was supported by Post Doctoral fellowship bestowed by TCG Crest. The second author was supported by the
Council of Scientific and Industrial Research grant no. 09/1248(0004)/2019-EMR-I, Ministry of Human Resource Development, Government of India.

\bibliographystyle{plain}

\end{document}